\documentclass[10 pt, conference]{ieeeconf}  

\IEEEoverridecommandlockouts                              

\usepackage{graphicx}
\usepackage{epsfig}
\usepackage{amsmath}

\usepackage{amsthm}

\usepackage{amssymb}
\usepackage[noadjust]{cite}
\usepackage{bm}
\usepackage{subfigure}
\usepackage{acronym}
\usepackage{paralist}
\usepackage{float}
\usepackage{epstopdf}
\usepackage{algorithm}
\usepackage{algpseudocode}
\usepackage{mathtools}
\usepackage{multicol} 
\usepackage{accents}
\newcommand{\subparagraph}{}
\usepackage{titlesec}

\newtheorem{define}{Definition}
\newtheorem{theorem}{Theorem}
\newtheorem{lemma}{Lemma}
\newtheorem{corr}{Corollary}
\newtheorem{remark}{Remark}
\newtheorem{assume}{Assumption}

\newcommand{\R}{\mathbb R}

\newcommand{\eps}{\epsilon}

\newcommand{\bmx}[1]{\begin{bmatrix}#1\end{bmatrix}} 
\newcommand{\bkt}[1]{\left[#1\right]} 
\newcommand{\pth}[1]{\left(#1\right)} 
\newcommand{\brc}[1]{\left \{#1\right \}} 
\newcommand{\nrm}[1]{\left \lVert#1\right \rVert} 
\newcommand{\ip}[1]{\left \langle #1 \right \rangle}

\newcommand{\bmxs}[1]{\begin{bsmallmatrix}#1\end{bsmallmatrix}}

\DeclarePairedDelimiter{\ceil}{\lceil}{\rceil}
\DeclarePairedDelimiter{\floor}{\lfloor}{\rfloor}
\DeclarePairedDelimiter{\abs}{\lvert}{\rvert}

\newcommand{\rarr}{\rightarrow} 

\makeatletter
\let\oldceil\ceil
\def\ceil{\@ifstar{\oldceil}{\oldceil*}}
\makeatother
\makeatletter
\let\oldfloor\floor
\def\floor{\@ifstar{\oldfloor}{\oldfloor*}}
\makeatother
\makeatletter
\let\oldnorm\norm
\def\norm{\@ifstar{\oldnorm}{\oldnorm*}}
\makeatother
\makeatletter
\let\oldabs\abs
\def\abs{\@ifstar{\oldabs}{\oldabs*}}
\makeatother

\makeatletter
\let\NAT@parse\undefined
\makeatother

\usepackage{hyperref}

\makeatletter
\hypersetup{colorlinks=true}
\AtBeginDocument{\@ifpackageloaded{hyperref}
  {\def\@linkcolor{blue}
  \def\@anchorcolor{red}
  \def\@citecolor{red}
  \def\@filecolor{red}
  \def\@urlcolor{red}
  \def\@menucolor{red}
  \def\@pagecolor{red}
\begingroup
  \@makeother\`%
  \@makeother\=%
  \edef\x{%
    \edef\noexpand\x{%
      \endgroup
      \noexpand\toks@{%
        \catcode 96=\noexpand\the\catcode`\noexpand\`\relax
        \catcode 61=\noexpand\the\catcode`\noexpand\=\relax
      }%
    }%
    \noexpand\x
  }%
\x
\@makeother\`
\@makeother\=
}{}}
\makeatother

\title{\LARGE \bf
Strong Invariance Using Control Barrier Functions: \\ A Clarke Tangent Cone Approach*
}

\author{James Usevitch, Kunal Garg, and Dimitra Panagou
\thanks{*The authors would like to acknowledge the support of Air Force Office of Scientific Research under award number FA9550-17-1-0284, the Automotive Research Center (ARC) in accordance with Cooperative Agreement W56HZV-14-2-0001 U.S. Army TARDEC in Warren, MI, and of the Award No W911NF-17-1-0526.}
\thanks{James Usevitch, Kunal Garg, and Dimitra Panagou are with the Aerospace Engineering Department at the University of Michigan, Ann Arbor, USA.{\tt\small \{usevitch,kgarg,dpanagou\}@umich.edu}}%
}

\newtheorem{Example}{Example}

\newtheorem{prop}{Proposition}

\newcommand{\cvxc}{\overline{\text{co}}}
\newcommand{\prj}{\textup{proj}}

\begin{document}

\maketitle
\thispagestyle{empty}
\pagestyle{empty}

\begin{abstract}

Many control applications require that a system be constrained to a particular set of states, often termed as safe set. A practical and flexible method for rendering safe sets forward-invariant involves computing control input using Control Barrier Functions and Quadratic Programming methods. Many prior results however require the resulting 
control input to be continuous, which requires strong assumptions or can be difficult to demonstrate theoretically.
In this paper we use differential inclusion methods to show that simultaneously rendering multiple sets invariant can be accomplished using a discontinuous control input.
We present an optimization formulation which computes such control inputs and which can be posed in multiple forms, including a feasibility problem, a linear program, or a quadratic program. In addition, we discuss conditions under which the optimization problem is feasible and show that any feasible solution of the considered optimization problem which is measurable renders the multiple safe sets forward invariant.
\end{abstract}

\section{Introduction}
Safety considerations such as maintaining a safe distance from static or dynamic obstacles for systems like robots, unmanned aerial vehicles, and autonomous cars is a critical concern in modern control theory. Safety requirements and other system objectives, such as confining the system trajectories to remain in a desired operating set, can be modelled as set invariance constraints where the objective is to guarantee that state trajectories remain within specified subsets of the state space under the closed-loop dynamics. Among other approaches, control barrier functions (CBFs) have been studied by many researchers to establish forward invariance of safe sets, thereby guaranteeing safety  and other system objectives are achieved \cite{romdlony2016stabilization,barry2012safety,ames2014control,ames2019control}. 
More recently, in \cite{ames2014control,ames2017control}, conditions using zeroing control barrier functions (ZCBF) are presented to ensure forward invariance of a desired set. 
Other authors have used CBFs to design control input using closed-form expressions that resemble Sontag's formula, e.g., \cite{romdlony2016stabilization}. 

For certain classes of nonlinear systems it can be difficult in general to find closed-form expressions for control inputs that render particular safe sets invariant.
The authors in \cite{ames2014control,li2018formally,ames2017control,glotfelter2017nonsmooth,glotfelter2018boolean} have explored online optimization methods of utilizing CBFs in control design, where typically, a quadratic program (QP) is set up to compute the control input at every point in the state space. In these works the CBF inequalities take form of the linear constraints in the QP. Since the QP needs to be solved pointwise in the state space, it becomes a parameteric optimization problem where the state variable acts as a parameter. The authors in \cite{fiacco1976sensitivity} studied parameteric convex optimization problems, and showed that the solution of the optimization is continuously differentiable if the objective function and the constraints functions are twice continuously differentiable, and strict complementary slackness holds. These conditions are relaxed in \cite{robinson1974perturbed}, where only continuity of these functions is assumed to guarantee that the solution of the parameteric convex optimization problem is a continuous function of the parameter.

In the particular context of control design using QPs, the authors in \cite{ames2017control} showed Lipschitz continuity of the solution of a QP under the assumption that the objective function and the functions defining the constraints in the QP are locally Lipschitz continuous, in the absence of control input constraints (see also \cite{xu2015robustness}). Under similar assumptions, the authors in \cite{glotfelter2017nonsmooth} show that the solution of QP is guaranteed to be Lipschitz continuous (in the absence of input constraints) if the CBF constraints are inactive, i.e., the constraints are satisfied with strict inequality at the optimal solution $z^*$. However demonstrating the Lipschitz continuity of the optimal solution for more general conditions, including when input constraints are incorporated, can be a nontrivial task.

The topic of guaranteeing set invariance under a possibly discontinuous control input has been studied for decades \cite{clarke1995qualitative,aubin2009set,clarke2008nonsmooth,cortes2008discontinuous}. Only somewhat recently has some of this theory been applied to set invariance using CBFs \cite{emam2019robust, glotfelter2017nonsmooth, glotfelter2018boolean}. 
In \cite{glotfelter2018boolean} the forward-invariance of multiple safety sets is considered under a discontinuous control input. Their methods involve incorporating multiple CBFs into a single nonsmooth function and then utilizing the generalized gradient and set-valued Lie derivative to demonstrate forward invariance. This methodology requires computationally tracking the notion of almost active gradients and considering set-valued inner products to generate the required control inputs.



This paper presents a different approach to guaranteeing strong invariance of multiple composed sets as compared to prior literature. 
The first contribution of this paper is to guarantee the simultaneous forward invariance of multiple subsets of the state space using CBFs and incorporating control input constraints. Unlike prior work, we approach the problem using the notions of Clarke tangent cones and transversality. We demonstrate that a constrained control input simultaneously rendering these subsets invariant can be generated by simply solving a feasibility problem with compact linear constraints. The control input is only required to be Lebesgue measurable and is not required to be continuous. In contrast to \cite{glotfelter2017nonsmooth, glotfelter2018boolean}, we demonstrate conditions under which the set-valued map of feasible controls rendering the composed sets strongly invariant is not only upper semicontinuous but also locally Lipschitz on a specified domain.


Our second main contribution is formulating a general convex optimization problem which computes control inputs that simultaneously render multiple subsets invariant. This optimization problem takes the form of a feasibility problem, with special cases being a Linear Program (LP) and QP. In contrast to \cite{glotfelter2017nonsmooth, glotfelter2018boolean} we show that under certain assumptions the proposed optimization problem is feasible, even in the presence of control input constraints. The feasibility of the optimization problem is shown to be sufficient to guarantee forward-invariance of multiple safe sets without requiring a continuity property of its solution as a function of the system states.

 
\section{Notation}
\label{sec:notation}
The boundary of a set $S \subset \R^n$ is denoted $\partial S$.
The closed convex hull of $S$ is $\cvxc(S)$.
The distance function associate with the set $S$ at $x \in \R^n$, denoted $d_S(x)$, is defined as $d_S(x) = \inf\brc{\nrm{x - s} : s \in S}$.
The polar cone of the set $S$ is the set $S^\circ = \brc{y \in \R^n : \ip{y,x} \leq 0\ \forall x \in S}$.

For notational brevity, given subsets $U \subset \R^{n \times m}$, $V \subset \R^{m \times p}$ the set-valued matrix product is defined as $U V = \{AB : A \in U,\ B \in V \} \subset \R^{n \times p}$. The Minkowski sum is denoted $U + V = \{A + B : A \in U,\ B \in V \}$. The Minkowski difference is $U_1 - U_2 = (U_1^c + U_2)^c$, where the set complement denotes $U^c = \R^{n \times m} \backslash U$.
Given a function $f : \R^{n \times m} \rarr \R^{m \times p}$, we denote $f(U) = \{f(A) : A \in U \}$.
The norm $\|\cdot\|$ in this paper refers to any sub-multiplicative matrix norm, i.e. $\|AB\| \leq \|A\| \|B\|$.

The open unit ball on a vector space $\R^{n \times m}$ is denoted 
The closed unit ball is denoted $\bar{B}^{{n \times m}}(0,1) = \cvxc(B^{n \times m}(0,1))$. The unit ball will be denoted as simply $B(0,1)$ when the dimensions are clear from the context.

The gradient of a continuously differentiable function $h : \R^n \rarr \R$ is denoted $\frac{\partial h}{\partial x}$, or in some cases as $\nabla h$. We use $h\in \mathcal C^{1,1}_{loc}$ to denote a continuously differentiable function, whose gradient $\nabla h$ is locally Lipschitz continuous.
The Lie derivative of a continuously differentiable function $h : \R^n \rarr \R$ along a vector field $f : \R^n \rarr \R^n$ is denoted $L_f h(x) \triangleq  \frac{\partial h}{\partial x} f(x)$.
A function $g : \R^n \rarr \R$ is \emph{locally bounded} on a set $D \subseteq \R^n$ if for all $x \in D$ there exists a neighborhood of $x$ denoted $U(x)$ and a constant $M \in \R$ such that $\nrm{g(z)} \leq M$ for all $z \in U(x)$.

\section{Problem Formulation}

Consider the control affine system
\begin{align}
\label{eq:controlaffine}
\begin{aligned}
    \dot{x}(t) &= f(x(t)) + g(x(t))u(t), \\
    u(t) &\in \mathcal{U} \subset \R^m\ \forall t \geq t_0.
\end{aligned}
\end{align}
The functions $f: \R^n \rarr \R^n$ and $g : \R^n \rarr \R^{n \times m}$ are assumed to be locally Lipschitz on $\R^n$. Without loss of generality, we let $t_0 = 0$. The set $\mathcal{U} \subset \R^m$ represents the set of \emph{feasible controls} for the system.

\begin{assume}
\label{assume:U}
The set $\mathcal{U}$ is a compact, convex polytope with $\textnormal{int}(\mathcal{U}) \neq \emptyset$ which has the form
\begin{align}
\begin{aligned}
\label{eq:upolytope}
\mathcal{U} &= \{u \in \R^m : A_u u \leq b_u \}, \\ A_u &\in \R^{p \times m},\ b_u \in \R^{p \times 1}
\end{aligned}
\end{align}
where $A_u$, $b_u$ are constant.
\end{assume}
Constraints of this form are common in prior literature \cite{ames2016control,cortez2019control,ames2017control,ames2014rapidly}.

\begin{Example}
    A specific example of control constraints satisfying Assumption \ref{assume:U} is bounding the input by an infinity norm, e.g. $\nrm{u}_\infty \leq u_{\max} \in \R$. This can be expressed in the form of \eqref{eq:upolytope} by setting
    $A_u = I_{m \times m} \otimes \bmxs{1 \\ -1}$,
        $b_u = (u_{\max}) \bm 1_{2m}$.
\end{Example}


The objective of the system \eqref{eq:controlaffine} is to compute a 
control input in order to simultaneously guarantee satisfaction of multiple set invariance constraints. 
The precise definition of strong invariance is given in Definition \ref{def:stronginv} below.
For the sake of generality, we give the definition in terms of differential inclusions of the form $\dot{x} \in F(x)$, which includes single-valued functions $F(x) = \{f(x,u) \}$ as a special case.
An overview of concepts related to differential inclusion theory is given in the Appendix, Section \ref{sec:diffincl}.

\begin{define}[\cite{clarke1995qualitative}]
\label{def:stronginv}
    Consider a differential inclusion $\dot{x}(t) \in F(x(t))$ and let $S \subset \R^n$. The system pair $(S,F)$ is said to be strongly invariant if \textbf{\textup{all}} trajectories of the system $x(\cdot)$ with $x(0) \in S$ satisfy $x(t) \in S$ for all $t \geq 0$.
\end{define}

More specifically, it is required that the system \eqref{eq:controlaffine} satisfy a \textbf{composition} of set invariance constraints encoded by sets $S_i \subseteq \R^n$, $i=1,\ldots,N_h$. Each set $S_i$ is defined as the sublevel\footnote{It is also common in prior literature to define each $S_i$ in terms of \emph{superlevel} sets, e.g. \cite{ames2014control}.} set of a continuous function $h_i : \R^n \rarr \R$ as follows:
\begin{align}
\begin{aligned}
\label{eq:lotsofsets}
    S_i &= \{x \in \R^n : h_i(x) \leq 0 \}, \\
    \textnormal{int}(S_i) &= \{x \in \R^n : h_i(x) < 0\}, \\
    \partial S_i &= \{x \in \R^n : h_i(x) = 0 \}.
    \end{aligned}
\end{align}

\noindent To characterize the properties of each $h_i$, we will use the notion of \textit{strict} CBFs: 

\begin{define}
The continuously differentiable function $h : \R^n \rarr \R$ is called a \textbf{strict} CBF for the set $S \subset \R^n$ defined as $S = \{x\; |\; h(x)\leq 0\}$ if the following holds:
\begin{align}
\label{eq:necessary}
    \inf_{u \in \mathcal{U}} \bkt{L_f h(x) + L_g h(x) u} < 0\ \forall x \in \partial S,
\end{align}
where $f,g$ are defined as in \eqref{eq:controlaffine}.
\end{define}

Note that the authors in \cite{ames2017control} call $h$ a CBF if \eqref{eq:necessary} holds with a non-strict inequality. Although the condition in \eqref{eq:necessary} may be stronger than necessary when $u(t)$ is guaranteed to be continuous, the property in \eqref{eq:necessary} will be useful when guaranteeing set invariance without a continuous control input. 
It is worth noting that this condition is required in \cite[Prop. 3]{glotfelter2017nonsmooth} to guarantee forward-invariance using a control input defined as a solution of a QP.



\begin{assume}
\label{assume:necessary}
Each $h_i$ from \eqref{eq:lotsofsets} satisfies $h_i\in \mathcal C^{1,1}_{loc}$, and is a strict CBF.
\end{assume}

\begin{assume}
\label{assume:compactS}
Each set $S_i$ is compact.
\end{assume}

\begin{remark}
Assumption \ref{assume:compactS} will aid the analysis of preventing possible finite escape time behavior of solutions. 
One way to guarantee compactness of closed-loop trajectories is having the closed-loop trajectories required to reach a desired equilibrium point or a set. It is possible to encode such requirements using CLFs in the optimization framework \cite{ames2017control,li2018formally,garg2019prescribed}. Such convergence constraints 
will be investigated in future work.
\end{remark}

Before we present the main results, we review forward-invariance of a set for differential inclusions. Consider a differential inclusion
\begin{align}\label{eq:diff inc}
 \dot{x}(t) \in F(x(t)),   
\end{align}
where $F : \R^n \rarr \mathcal{P}(\R^n)$.
Let $S \subseteq \R^n$ and let $T_S(x)$ be the tangent cone to $S$ at $x$, as defined in Definition \ref{def:tangentcone}.

The following \emph{Standing Hypotheses} will be used in this paper:
\begin{define}\label{def:standinghyp}
The following conditions are termed the \emph{Standing Hypotheses}:
\begin{itemize}
    \item[a)] For every $x \in D \subseteq \R^n$, $F(x)$ is nonempty, compact, and convex
    \item[b)] $x \mapsto F(x)$ is upper semicontinuous;
    \item[c)] $F(x)$ is locally bounded; i.e. for all $x \in \R^n$ there exist $\epsilon, m > 0$ such that $\nrm{z} \leq m$ for all $z \in F(y)$, for all $y \in B(x,\eps)$.
\end{itemize}
\end{define}

The following fundamental theorem describes how strong invariance of a set can be achieved with respect to a system described by a differential inclusion. 

\begin{theorem}[Adapted from \cite{clarke1995qualitative}]
\label{thm:Clarkestrong}
Let $F$ be locally Lipschitz and suppose that $F$ satisfies the Standing Hypotheses (Definition \ref{def:standinghyp}). Then the following are equivalent for \eqref{eq:diff inc}:
\begin{itemize}
    \item[(1)] $F(x) \subseteq T_S(x)$ $\forall x \in S$;
    \item[(2)] $(S,F)$ is strongly invariant.
\end{itemize}
\end{theorem}

\begin{remark}
    In prior literature the condition \emph{c)} above is sometimes replaced by the following linear growth condition:
    \begin{itemize}
        \item[$c'$)] For certain constants $\gamma$ and $c$, and for all $x \in D \subseteq \R^n$,
        \begin{align}
            v \in F(x) \implies \|v\| \leq \gamma \|x\| + c.
        \end{align}
    \end{itemize}
    The condition $c')$ is a sufficient condition to ensure that the system does not exhibit finite escape time (\cite{clarke2008nonsmooth}, \emph{Notes and Comments, Ch. 4}). Other methods can be used however to guarantee that finite escape time is avoided.
\end{remark}

Given the sets $S_1,\ldots,S_{N_h}$ defined by functions $h_1,\ldots,h_{N_h}$, the purpose of this paper is to demonstrate a method of computing a measurable, possibly discontinuous control input $u$ which simultaneously renders the sets strongly invariant by using the result on strong invariance in Theorem \ref{thm:Clarkestrong}.


\section{Main Results}

In this section we demonstrate how the conditions of Theorem \ref{thm:Clarkestrong} can be satisfied by design through solving a feasibility problem.
We will approach this problem by designing a differential inclusion of the form
\begin{align}
    G(x) = \{f(x) + g(x)u : u \in K(x) \}
\end{align}
where the set-valued map $K : \R^n \rarr \mathcal{P}(\R^m)$ satisfies $K(x) \subseteq \mathcal{U}$ for all $x \in \R^n$. 
The behavior of \eqref{eq:controlaffine} under any Lebesgue measurable $u{\color{red}(t)}\in K(x)$ can then be studied by analyzing $G(x)$. This is a common method in the literature for considering all trajectories of a controlled system simultaneously \cite[Ch. 3, \S 15]{filippov2013differential}, \cite[Eq. (1.2)]{clarke1995qualitative}, \cite[Ch. 10]{aubin2009set}, \cite[Eq. (34)]{cortes2008discontinuous}.

\subsection{Invariance of a Single Set}

For simplicity of presentation, the first portion of our results will consider a system with only one set $S = \brc{x : h(x) \leq 0}$ to be rendered invariant. Considering multiple sets will then be analyzed in Section \ref{sec:multiplesets}.

We begin by defining the set-valued map $K(\cdot)$. In the prior work (see e.g., \cite{ames2017control}), the forward invariance of a single set was guaranteed by considering a locally Lipschitz continuous control input $u(t)$ within the set
\begin{align}
    \{u \in \mathcal{U} : L_fh(x) + L_gh(x)u \leq -\alpha(h(x)) \},
\end{align}
for all $t\geq 0$. Inspired by this method, consider the set-valued map
\begin{align}
\begin{aligned}
\label{eq:Kforreal}
        &K(x) = \brc{u \in \R^m : \bmx{A_S(x) \\ A_u} u \leq \bmx{b_S(x) \\ b_u}},
\end{aligned}
\end{align}
where $A_S : \R^n \rarr \R^{q \times m}$ and $b_S : \R^n \rarr \R^{q}$ are defined in this case as
\begin{align}
    A_S(x) &= L_g h(x),\\
    b_S(x) &= -\alpha(h(x)) - L_f h(x).
\end{align}
Here, $\alpha(\cdot)$ is an extended class-$\mathcal{K}_\infty$ function which is locally Lipschitz on $\R$.
Note that $A_S(x)$ and $b_S(x)$ are each locally Lipschitz on $\R^n$. This holds since by \eqref{eq:controlaffine} and Assumption \ref{assume:necessary} the functions $f$, $g$, $\frac{\partial h}{\partial x}$ are locally Lipschitz on $\R^n$, and the sums and products of locally Lipschitz functions on $\R^n$ are also locally Lipschitz on $\R^n$. The set $K(x)$ can be considered as the \emph{feasible set} of the combined set invariance and control input constraints for $x\in S$.
In preparation for later results we define the set $\Omega \subset \R^n$ as
\begin{align}
\label{eq:Omega}
    \Omega = \{x \in \R^n: \textnormal{int}(K(x)) \neq \emptyset \}.
\end{align}
Note that under Assumption \ref{assume:necessary}, it holds that $\partial S \subset \Omega$ and $\textnormal{int}(S \cap \Omega) \neq \emptyset$. The following result demonstrates conditions under which the interior of $K$ is a locally Lipschitz set-valued map.

\begin{lemma}
\label{lem:KLips}
    Let $K$ be defined as in \eqref{eq:Kforreal}.
    If $A_S$, $b_S$ are locally Lipschitz on a bounded, open set $D \subseteq \Omega$, then $\textup{int}(K)$ is locally Lipschitz continuous on $D$.
\end{lemma}

\begin{proof}
    The proof will employ the result in \cite[Prop. 2.14]{freeman2008robust}, which is included as Proposition \ref{prop:Freeman} in the Appendix for convenience.
    Define $\mathcal{U}(x) = \{u : A_u u \leq b_u\}$ which has nonempty, compact, convex values for all $x \in D \subseteq \R^n$.
    Observe that $K(x)$ is equivalent to $K(x) = \{u : A_u u \leq b_u\} \cap \{u : A_S(x) u \leq b_S(x) \}$.  Let $\phi_j(x,u) = A_{S,j}(x)u - b_{S,j}(x)$ where $A_{S,j}(x)$ is the $j$th row of $A_S(x)$. 
    Since $A_S$ and $b_S$ are locally Lipschitz on $D$, for fixed $u \in \R^m$ each $\phi_j(\cdot,u)$ is locally Lipschitz on $D$.
    Note that $\phi_j(x,u)$ is affine in $u$ for fixed $x$ for all $j = 1,\ldots,q$, and therefore $\phi_j(x,\cdot)$ is convex and locally Lipschitz on $\R^m$. By Lemma \ref{lem:twolocalLipschitz} (see Appendix) this implies that each $\phi_j : D \times \R^m \rarr \R$ is locally Lipschitz on $D \times \R^m$.
    Similarly, let $\phi_k(x,u) = A_{u,k} u - b_{u,k}$ for all $k=q+1,\ldots,q+p$. The function $\phi_k(\cdot,u)$ for fixed $u$ is constant in $x$ and therefore locally Lipschitz on $D$. The function $\phi_k(x,\cdot)$ for fixed $x$ is affine in $u$ and therefore convex and locally Lipschitz on $\R^m$. By Lemma \ref{lem:twolocalLipschitz} this implies that the functions $\phi_k : D \times \R^m \rarr \R$ are locally Lipschitz on $D \times \R^m$ for $k=q+1,\ldots,q+p$.
    
    
    Define $\phi^*(x,u) = \max_{j} \phi_j(x,u)$. Since the pointwise maximum over convex functions is convex \cite[Sec. 3.2.3]{boyd2004convex}, the function $\phi^*(x,\cdot)$ is therefore convex in $u$ for all fixed $x$.
    We next prove that $\phi^*$ is locally Lipschitz on $D \times \R^m$.
    Define the neighborhood $U^*(x,u) = \bigcap_i U_i(x,u)$ where each $U_i(x,u)$ is the Lipschitz neighborhood for each $\phi_i(x,u)$.
    Observe that since the range of each $\phi^j$ and $\phi^*$ is $\R$, for any $\bmxs{x_1 \\ u_1}, \bmxs{x_2 \\ u_2} \in U^*(x,u) $ we have
    \begin{align*}
        &|\phi^*(x_1,u_1) - \phi^*(x_2,u_2)| = |\max_i \pth{\phi_i(x_1,u_1)} - \\
        & \hspace{4.5cm} \max_i \pth{\phi_i(x_2,u_2)}| \\
        &\leq \max_i \pth{|\phi_i(x_1,u_1) - \phi_i(x_2,u_2) |} \\
        &\leq \max_i\pth{L_i(x,u)} \nrm{\bmx{x_1 \\ u_1} - \bmx{x_2 \\ u_2}}
    \end{align*}
    The function $\phi^*$ is therefore locally Lipschitz on $D \times \R^m$. From Proposition \ref{prop:Freeman} in Appendix, the set-valued mapping $\bar{F}(x) = \{z \in \mathcal{U}(x) : \phi^*(x,u) < 0 \}$ is therefore locally Lipschitz on $D$. However since each $\phi_j(x,u) < 0$ if and only if $A_{S,j}(x) u - b_{S,j}(x) < 0$ for $j = 1,\ldots,q$ and each $\phi_k(x,u) < 0$ if and only if $A_{u,k} u - b_{u,k} < 0$ for $k = 1,\ldots,p$, we therefore have $\bar{F}(x) = \text{int}(K(x))$ which concludes the proof.
\end{proof}

The previous Lemma demonstrated that the function $\textnormal{int}(K)$ is locally Lipschitz on any bounded, open subset $D \subset \Omega$. However, to construct a $G(x)$ which is compact we will need to consider a \emph{closed} bounded set-valued map which is a subset of $\textnormal{int}(K(x))$ at every point $x$, and also locally Lipschitz on $D$.
Towards this end, consider a bounded, open domain $D \subset \Omega$ and let $0 < \gamma < \inf_{x \in D} R_C(K(x))$, where 
\begin{align}
\label{eq:chebyshev}
    R_C(S) \triangleq \sup_{u \in S} d_{S^c}(u),\ S^c = \R^m \backslash S,
\end{align}
is the radius of the largest ball which can be inscribed in $S \subset \R^m$ \cite[Sec 8.5]{boyd2004convex}.
We define the $\gamma$ contraction of $K(x)$ as
    \begin{align}
    \label{eq:Kgamma}
        K_{\gamma}(x) &= \textnormal{int}(K(x)) - \gamma B(0,1), \\
        &= \{u \in K(x) : d_{K^c}(u) \geq \gamma \},\ K^c = \R^m \backslash K(x), \nonumber 
    \end{align}
    where $\textnormal{int}(K(x)) - \gamma B(0,1)$ denotes the Minkowski difference operation defined in Section \ref{sec:notation}.
Note that $K_\gamma(x)$ is closed for all $x$ in $D \subset \Omega$, since $B(0,1)$ denotes the open ball of radius 1. The parameter $\gamma$ is chosen in a such a way to guarantee that $K_\gamma(x)$ is nonempty for all $x \in D$.


\begin{lemma}
\label{lem:kgamma}
    If $\textnormal{int}(K)$ is locally Lipschitz on a bounded, open set $D \subset \Omega$, then for any $0 < \gamma < \inf_{x \in D} R_C(x)$ it holds that $K_{\gamma}$ is locally Lipschitz on $D$. 
\end{lemma}

\begin{proof}
    Since $\textnormal{int}(K)$ is nonempty and locally Lipschitz on $D$, then for all $x\in D$ there exists a domain $U(x)$ and constant $L = L(x)$ such that $\textnormal{int}(K(y)) \subseteq \textnormal{int}(K(z)) + L\nrm{y-z}B(0,1)$. Taking a Minkowski difference from both sides yields
    \begin{align*}
        \textnormal{int}(K(y)) - \gamma B(0,1) &\subseteq \textnormal{int}(K(z)) - \gamma B(0,1) + \\ 
        & \hspace{3em} L\nrm{y-z} B(0,1) \\
        K_{\gamma}(y) &\subseteq K_{\gamma}(z) + L\nrm{y-z} B(0,1).
    \end{align*}
    The result follows. Note that since $R_C(x) > \gamma$ for all $x \in D \subset \Omega$, $K_\gamma(x)$ is never empty in $D$.
\end{proof}


\noindent The feasible set $K_{\gamma}(x)$ can be used to construct the following set-valued map $G_{\gamma} : \R^n \rarr \mathcal{P}(\R^n)$:
\begin{align}
\label{eq:Geex}
    &G_{\gamma}(x) = \brc{v \in \R^n :  v = f(x) + g(x) u,\  u \in K_{\gamma}(x)}.
\end{align}

\noindent We next prove several properties about the set-valued map $G_{\gamma}$ in order to show that it satisfies the conditions of Theorem \ref{thm:Clarkestrong}.
First, we show that $G_\gamma$ is locally Lipschitz on any open, bounded subset of $\Omega$.

\begin{theorem}
\label{thm:One}
Let $D \subset \Omega$ be a bounded, open set.
Then the set-valued map $G_{\gamma}$ from \eqref{eq:Geex} is locally Lipschitz on $D$.
\end{theorem}

\begin{proof}
From Lemma \ref{lem:kgamma}, we have that the local Lipschitzness of $K$ on $D$ implies local Lipschitzness of $K_\gamma$ on $D$. 
Recall that the functions $A_S, b_S$ are locally Lipschitz continuous since $f, g$ are locally Lipschitz continuous, and $h\in C^{1,1}_{loc}$. Thus, from Lemma \ref{lem:KLips}
the set-valued map $K$ is locally Lipschitz.
These and the preceding arguments imply that $K$ and therefore $K_{\gamma}$ are locally Lipschitz on $D \subset \Omega$.

Per \eqref{eq:Geex}, the set-valued mapping $G_{\gamma}(x)$ is equal to the image of $K_{\gamma}(x)$ under the affine mapping $f(x) + g(x) u$, $u \in K_{\gamma}(x)$. 
We must next show that for all $x \in \R^n$ there exists a neighborhood $U(x)$ and constant $L = L(x)$ such that $G_{\gamma}(x_2) \subseteq G_{\gamma}(x_1) + L \nrm{x_{12}} \bar{B}^{n\times 1}(0,1)$ for all $x_1,x_2 \in U(x)$, where $x_{12} = x_1 - x_2$.
For brevity, we abbreviate $\bar{B}^{n \times m} = \bar{B}^{n \times m}(0,1)$.
Since $f$, $g$, and $K_{\gamma}$ are locally Lipschitz on $D$, it holds that there exist neighborhoods $U_1(x)$, $U_2(x)$, $U_3(x)\subset \mathbb R^n$ of $x$ and constants $L_1 = L_1(x)$, $L_2 = L_2(x)$, $L_3 = L_3(x)$ such that
\begin{align*}
    \{f(x_2)\} &\subset f(x_1) + L_1\nrm{x_{12}}\bar{B}^{n\times 1}\, \quad\forall x_1,x_2 \in U_1(x), \\
    \{g(x_2) \} &\subset g(x_1) + L_2 \nrm{x_{12}} \bar{B}^{n\times m}\, \quad \forall x_1,x_2 \in U_2(x), \\
    K_{\gamma}(x_2) &\subseteq K_{\gamma}(x_1) + L_3 \nrm{x_{12}} \bar{B}^{m \times 1}\, \quad \forall x_1, x_2 \in U_3(x).
\end{align*}
By \eqref{eq:Geex} and from the local Lipschitzness of $K_{\gamma}$ we have the following:\small{
\begin{align}
    \hspace{-5pt} G_{\gamma}(x_2) =& f(x_2) + g(x_2) K_{\gamma}(x_2) \nonumber \\
    \subseteq & (f(x_1) + L_1 \nrm{x_{12}} \bar{B}^{n\times 1}) + \nonumber\\
    &  (g(x_1) + L_2 \nrm{x_{12}} \bar{B}^{n \times m})(K_{\gamma}(x_1) + L_3\nrm{x_{12}}\bar{B}^{m \times 1}) \nonumber\\
    \subseteq & (f(x_1) + g(x_1) K_{\gamma}(x_1)) + \nrm{x_{12}} \Big[ L_1 \bar{B}^{n \times 1} \nonumber\\
    & + L_2 \bar{B}^{n \times m} K_{\gamma}(x_1) + L_3 g(x_1) \bar{B}^{m \times 1} \nonumber\\ 
    & + L_2 L_3 \nrm{x_{12}} \bar{B}^{n \times m} \bar{B}^{m \times 1} \Big] \label{eq:reallylong}
\end{align}}\normalsize
By Lemmas \ref{lem:setmult}, \ref{lem:setmatrixprod} in the Appendix, it holds that
\begin{align}
    \bar{B}^{n \times m} K_{\gamma}(x_1) &\subseteq \nrm{v_{\max}} \bar{B}^{n \times 1}, 
\end{align}
where $v_{\max} \triangleq \underset{v \in K_{\gamma}(x_1)}{\arg\sup} \nrm{v}$ exists and is finite since $K_{\gamma}(x_1)$ is compact. Using  Lemmas \ref{lem:setmult}, \ref{lem:setmatrixprod} also yields
\begin{align}
    g(x_1) \bar{B}^{m \times 1} &\subseteq \nrm{g(x_1)} \bar{B}^{n \times 1} \\
    \bar{B}^{n \times m} \bar{B}^{m \times 1} & \subseteq \bar{B}^{n \times 1}
\end{align}
Note that since $g$ is locally Lipschitz on $\R^n$, $\nrm{g(x)}$ is locally bounded\footnote{The concept of local boundedness is described in Section \ref{sec:notation}.}
on $\R^n$. Using these results, \eqref{eq:reallylong} can be simplified to yield
\begin{align}
    G_{\gamma}(x_2) &\subseteq f(x_1) + g(x_1) K_{\gamma}(x_1)  + L \nrm{x_{12}} \bar{B}^{n \times 1} \\
    &\subseteq G_{\gamma}(x_1) + L \nrm{x_{12}} \bar{B}^{n \times 1},
\end{align}
where $L = L_1 + L_2 \nrm{v_{\max}} + L_3 \nrm{g(x_1)} + L_2 L_3 \nrm{x^*_{12}}$,  $x^*_{12} = \underset{x_1,x_2 \in U(x)}{\arg\sup} \nrm{x_{12}}$ and $U(x)= U_1(x) \cap U_2(x) \cap U_3(x)$. We can therefore conclude that $G_{\gamma}$ is locally Lipschitz on $D$.
\end{proof}

\noindent Our next result demonstrates that $G_\gamma$ satisfies the Standing Hypotheses from Definition \ref{def:standinghyp}.

\begin{lemma}
\label{lem:Two}
The set-valued map $G_{\gamma}$ from \eqref{eq:Geex} satisfies the Standing Hypotheses from Definition \ref{def:standinghyp} (see Appendix) for all $x$ in any bounded, open domain $D \subset \Omega$.
\end{lemma}

\begin{proof}
First, recall that $G_{\gamma}(x)$ is equal to the image of $K_{\gamma}(x)$ under the affine mapping $f(x) + g(x) u$. 
By definition of $\gamma$, i.e. $0 < \gamma < \inf_{x \in D} R_C(x)$, $K_{\gamma}$ is nonempty on $D$.
Note that $K_\gamma(x) \subset K(x) \subseteq \mathcal{U}$, implying that $K_\gamma(x)$ is bounded for all $x$.
By convexity of $K(x)$ and the definition of $K_\gamma(x)$ we have that $K_{\gamma}(x)$ is closed and convex.
Since $K_{\gamma}(x)$ is therefore convex, compact, and nonempty for all $x \in \Omega$, $G_{\gamma}(x)$ is therefore also convex, compact, and nonempty for all $x \in \Omega$ \cite[Sec. 2.3.2]{boyd2004convex}. Since $G_{\gamma}$ is locally  Lipschitz on $ \Omega$ by Theorem \ref{thm:One}, $G_{\gamma}(x)$ is therefore upper semicontinuous for all $x \in \Omega$.
Finally, $G_{\gamma}$ being locally Lipschitz on $\Omega$ implies that $G_{\gamma}(x)$ is locally bounded for all $x \in \Omega$.
This can be seen by choosing $\epsilon$ such that $B(x,\epsilon) \subseteq U(x)$, where $U(x)$ is the Lipschitz neighborhood for $G_{\gamma}(x)$, and choosing $m = \sup_{y \in G_{\gamma}(x)}\nrm{y}
+ \epsilon L(x)$ where $L(x)$ is the Lipschitz constant for $G_{\gamma}(x)$ at $x$. Note that $\sup_{y \in G_{\gamma}(x)}\nrm{y}$ is finite for all $x \in \Omega$ since $G_{\gamma}(x)$ is compact for all $x \in \Omega$.
\end{proof}

\begin{lemma}
\label{lem:Three}
The set-valued map $G_{\gamma}$ from \eqref{eq:Geex} satisfies $G_{\gamma}(x) \subseteq T_S(x)$ for all $x \in \R^n$.
\end{lemma}

\begin{proof}
    First, observe that for all $x \in \R^n \backslash \Omega$ we trivially have $G_{\gamma}(x) = \emptyset \subset T_S(x)$.
    Next, observe that for all $x \in \textnormal{int}(S) \cap \Omega$, $T_S(x) = \R^n$ implying $G_{\gamma}(x) \subset T_S(x)$. 
    We now focus on the set $\partial S \cap \Omega$. Since $h$ is continuously differentiable, by Lemma \ref{lem:cones} we have that $T_S(x) = \brc{v \in \R^n : \ip{v,\frac{\partial h(x)}{\partial x}} \leq 0}$ for all $x \in \partial S$. Observe that by the definition of $G_{\gamma}$, every $v \in G_{\gamma}(x)$ satisfies $\ip{\frac{\partial h(x)}{\partial x}, v} \leq 0$, which implies $G_{\gamma}(x) \subseteq T_S(x)$.
\end{proof}

Using Theorem \ref{thm:One} and Lemmas \ref{lem:Two}-\ref{lem:Three}, we can now prove the first main result of the paper that concerns the invariance of the set $S$ for the closed-loop trajectories of \eqref{eq:controlaffine}.

\begin{theorem}
\label{thm:main1}
    Consider the system 
    \begin{align} 
    \label{eq:theoremsystem}
    \dot{x}(t) \in G_{\gamma}(x(t)).
    \end{align}
    Let $S$ be a set defined as in \eqref{eq:lotsofsets} for some strict control barrier function $h$.
    Let $x(\cdot)$ be any trajectory of \eqref{eq:theoremsystem} under a Lebesgue measurable control input $u(\cdot)$ with $x_0 = x(0) \in \textnormal{int}(S \cap \Omega)$. Let $[0,T(x_0))$ be the (possibly empty) maximal interval such that $x(t) \in \textnormal{int}(\Omega)$ for all $t \in [0,T(x_0))$. Then $x(t) \in S$ for all $t \in [0,T(x_0))$.
\end{theorem}

\begin{proof}
Recall that Assumption \ref{assume:necessary} implies that $\partial S \subset \Omega$ and that $\textnormal{int}(S \cap \Omega) \neq \emptyset$.
By Theorem \ref{thm:One} and Lemma \ref{lem:Two}, $G_{\gamma}$ satisfies the Standing Hypotheses from Definition \ref{def:standinghyp} and is locally Lipschitz on $\textnormal{int}(S \cap \Omega)$, which guarantees existence of solutions to \eqref{eq:theoremsystem}.
By Lemma \ref{lem:Three}, we have $G_{\gamma}(x) \subseteq T_S(x)$ for all $x \in \R^n$.
By Theorem \ref{thm:Clarkestrong}, the trajectory $x(t)$ will remain in $S$ as long as $x(t) \in \textnormal{int}(\Omega)$; therefore $x(t) \in S$ for all $t \in [0,T(x_0))$.
\end{proof}

Theorem \ref{thm:main1} considers the general case where $(S \setminus \Omega) \neq \emptyset$; i.e. there may exist interior points of $S$ which are not in $\Omega$.\footnote{Recall that under Assumption \ref{assume:necessary} it holds that $\partial S \subset \Omega$ and $\text{int}(S \cap \Omega) \neq \emptyset$.} Since the set-valued mapping $G_{\gamma}$ satisfies the conditions of Theorem \ref{thm:Clarkestrong} only on bounded, open subsets of $\Omega$, strong invariance cannot be guaranteed by Theorem \ref{thm:Clarkestrong} for any trajectory which leaves $\Omega$. However in the case that $S \subseteq \Omega$, i.e. the interior of $K(x)$ is non-empty for all $x\in S$, the following corollary shows that the system pair $(S,G)$ is strongly invariant. 

\begin{corr}
\label{cor:main1}
    Under the hypotheses of Theorem \ref{thm:main1}, suppose there exists a bounded, open domain $D \subseteq \Omega$ such that $S \subset D$. If $x_0 = x(0) \in S$ then any Lebesgue measurable control input $u(t) \in K_{\gamma}(x(t))$ renders the pair $(S,G_{\gamma})$ strongly invariant; i.e. $x(t) \in S$ for all $t \geq 0$.
\end{corr}

\begin{proof}
    Since $S \subset D \subseteq \Omega$, by Theorem \ref{thm:One} and Lemma \ref{lem:Two}, $G_{\gamma}(x)$ satisfies all the Standing Hypotheses from Definition \ref{def:standinghyp} for all $x \in S$ and is locally Lipschitz for all $x\in S$. 
    Since $S$ is also compact by Assumption \ref{assume:compactS}, the interval of existence for all solutions to \eqref{eq:theoremsystem} is $[0,\infty)$.
    By Lemma \ref{lem:Three}, we have $G_{\gamma}(x) \subseteq T_S(x)$ for all $x \in \R^n$.
    The result follows by Theorem \ref{thm:Clarkestrong}.
\end{proof}

\subsection{Invariance of Multiple Sets}
\label{sec:multiplesets}
In this section, we discuss how to incorporate multiple safety requirements in an optimization framework and discuss conditions under which the resulting optimization problem is feasible. Consider the set of functions $h_i:\mathbb R^n\rightarrow \mathbb R$ defining the sets $S_{i} = \{x\; |\; h_i(x)\leq 0\}$, for $i =1, 2, \ldots, N_h$. Defining the composed set $S_I = \bigcap_{i=1}^{N_h} S_i$, we seek to render the set $S_I$ strongly invariant.
Recall that the functions $h_i$ satisfy Assumption \ref{assume:necessary}.
Incorporating general nonsmooth $h_i$ functions into this analysis will be considered in future work.

Similar to the previous section, we define the set-valued map $\widehat{K}(x)$ as
\begin{align}
\begin{aligned}
\label{eq:multiKforreal}
        &\widehat{K}(x) = \brc{u \in \R^m : \bmx{\hat{A}_S(x) \\ A_u} u \leq \bmx{\hat{b}_S(x) \\ b_u}}, \\
        &\hat{A}_S : \R^n \rarr \R^{q \times m},\ \hat{b}_S : \R^n \rarr \R^{q}.
\end{aligned}
\end{align}
In this case we define
\begin{align}
    \hat{A}_S(x) &= \bmx{L_g h_1(x) \\ \vdots \\ L_g h_{N_h}(x)}, \nonumber \\
    \hat{b}_S(x) &= \bmx{-\alpha_1(h_1(x)) - L_f h_1(x) \\ \vdots \\ -\alpha_{N_h}(h_{N_h}(x)) - L_f h_{N_h}(x)}, \label{eq:ABmultiple}
\end{align}
where each $\alpha_i(\cdot)$ is a extended class-$\mathcal{K}_\infty$ function which is locally Lipschitz on $\R$. Using $\widehat{K}(x)$ we define $\widehat{\Omega} = \{x \in \R^n : \text{int}(\widehat{K}(x)) \neq \emptyset  \} $.
Given a bounded, open domain $D \subset \widehat{\Omega}$, we also define $\widehat{K}_{\hat{\gamma}}(x)$ as
\begin{align}
    \widehat{K}_{\hat{\gamma}}(x) = \textnormal{int}(\widehat{K}(x)) - \hat{\gamma} B(0,1) 
\end{align}
where $\hat{\gamma}$ satisfies $0 < \hat{\gamma} < \inf_{x \in D} R_C(\widehat{K}(x))$ and $R_C(\cdot)$ is defined in \eqref{eq:chebyshev}.
The set-valued map $\widehat{G}_{\hat{\gamma}}(x)$ is defined as
\begin{align}
    \widehat{G}_{\hat{\gamma}}(x) = \brc{v \in \R^n :  v = f(x) + g(x) u,\  u \in \widehat{K}_{\hat{\gamma}}(x)}.
\end{align}

\noindent Using prior results, we can then show the following properties on $\widehat{G}_{\hat{\gamma}}(x)$.

\begin{lemma}
\label{lem:barG}
    Let $D \subset \widehat{\Omega}$ be a bounded, open set. Then the set-valued map $\widehat{G}_{\hat{\gamma}}$ is locally Lipschitz on $D$ and satisfies all the Standing Hypotheses from Definition \ref{def:standinghyp} for all $x \in D$. Furthermore, for all $x \in D$ we have
    \begin{align}
    \label{eq:Tintersect}
        \widehat{G}_{\hat{\gamma}}(x) \subseteq \bigcap_{i=1}^{N_h} T_{S_i}(x)
    \end{align}
\end{lemma}

\begin{proof}
    The result follows by using similar arguments as in Lemmas \ref{lem:KLips}-\ref{lem:Three} and Theorems \ref{thm:One}-\ref{thm:main1}.
    Note that by \eqref{eq:ABmultiple} each constraint $\hat{A}_{S,i}(x) u \leq \hat{b}_{S,i}(x)$ ensures that $\widehat{G}_{\hat{\gamma}}(x) \subseteq T_{S_i}(x)$.
\end{proof}

We are ultimately interested in rendering the composed set $S_I$ invariant by guaranteeing that $\widehat{G}_{\hat{\gamma}}(x) \subseteq T_{S_I}(x)$ for all $x \in S_I$. Note that in general $S_I$ may have a boundary that cannot be described by a $\mathcal{C}^1$ function. 
From \eqref{eq:Tintersect} the question remains as to whether $\bigcap_{i=1}^{N_h}T_{S_i}(x) \subseteq T_{S_I}(x)$ for all $x \in S_I$. If so, then by \eqref{eq:Tintersect} it holds that $\widehat{G}_{\hat{\gamma}}(x) \subseteq T_{S_I}(x)$.
Towards this end we review some mathematical preliminaries required to establish this condition. The concept of \textit{transversality} was explored in \cite{clarke2008nonsmooth} precisely to relate the intersection of tangent cones and the tangent cone of intersections of sets. Let $N_S(x)$ denote the normal cone of set $S$ at $x$. Note that since each $h_i \in \mathcal{C}^{1,1}$, by Lemma \ref{lem:cones} in the Appendix it holds that $N_{S_i}(x) = N^P_{S_i}(x)$ for all $x \in S_i$, for all $i =1,\ldots,N_h$.

\begin{define}
Transversality holds for the pair $(S_1, S_2)$ of two closed sets $S_1, S_2\subset\mathbb R^n$ if 
for all $x\in \partial S_1 \cap \partial S_2$ 
we have the following:
\begin{align}\label{eq: trans s1 s2}
    N_{S_1}(x) \cap (-N_{S_2}(x)) = \{0\},
\end{align}
\end{define}

\noindent Again, from \eqref{eq:Tintersect} we are interested in proving that $\bigcap_{i=1}^{N_h}T_{S_i}(x) \subseteq T_{S_I}(x)$ for all $x \in S_I$. The authors in \cite{clarke2008nonsmooth} presented the following implication of the transversality condition. 

\begin{lemma}\cite[pp 99]{clarke2008nonsmooth}\label{lemma trans clarke}
Let $S_1, S_2$ {\color{blue}be} defined as $S_i= \{x\; |\; h_i(x)\leq 0\}$, where $h_i\in \mathcal C^1$. If transversality condition \eqref{eq: trans s1 s2} holds for the pair $(S_1,S_2)$ then 
\begin{align}
    T_{S_1\cap S_2}(x) & = T_{S_1}(x)\cap T_{S_2}(x), \label{eq: T S1 S2 trans}\\ 
    N_{S_1\cap S_2}(x) & = N_{S_1}(x) + N_{S_2}(x), \label{eq: N S1 S2 trans} 
\end{align}
for all $x \in S_1 \cap S_2$.
\end{lemma}

Lemma \ref{lemma trans clarke} states the relation between the tangent and the normal cones of the intersection of two sets $S_1, S_2$, defined as zero sub-level sets of smooth functions, when the transversality condition holds. We can extend this result for arbitrary number of sets when pairwise transversality condition holds. Let $I:\mathbb R^n\rightarrow 2^{N_h}$ be the collection of indices of sets intersecting on their boundaries, defined as $I(x) = \{i\; |\; \exists j\neq i, h_i(x) = h_j(x) = 0\}$. 


\begin{lemma}\label{Lemma trans T cone}
If the transversality condition holds for the pair $(S_i, S_j)$ for all $i, j\in I(x)$, then 
\begin{align}\label{eq: transv S1 S2 SN}
    T_{\pth{\cap_i S_i}}(x) = \bigcap_{i=1}^{N_h} T_{S_i}(x),
\end{align}
holds for all $x \in (\bigcap_{i=1}^{N_h} S_i)$.
\end{lemma}
\begin{proof}
Note that the result holds trivially for any $x\in \bigcap_i \textnormal{int}(S_i)$. So, in the rest of the proof, we assume that $x$ is on the boundary of the considered set(s).  

Consider $i, j\in I(x)$ for some $x\in (\bigcap_{i=1}^{N_h} S_i)$. Since transversality holds for the pair $(S_i, S_j)$, we know that there does not exist $k> 0$ such that $\nabla h_i(x) + k\nabla h_j(x) = 0$, i.e., the vectors $\nabla h_i, \nabla h_j$ are not \textit{anti}-parallel.\footnote{This is true since \eqref{eq: trans s1 s2} does not hold when $\nabla h_i$ and $\nabla_j$ point in exactly opposite directions.} Now, the case when $\nabla h_i$ and $\nabla h_j$ are co-linear is trivial\footnote{This case is trivial because co-linearity of $\nabla h_i, \nabla h_j$ implies that their normal and tangent cones coincide.}, and thus, we focus on the case when neither the vectors $\nabla h_i, \nabla h_j$ are co-linear, nor they are anti-parallel. In other words, we focus on the case when transversality of $(S_i, S_j)$ implies that $\nabla h_i, \nabla h_j$ are linearly independent. 
Consider now another set $S_k$, whose normal cone is given by $N_{S_k} = \{y\; |\; y = c\nabla h_k, c\geq 0\}$ for all $x\in \partial S_{k}$. The normal cone of $S_1\cap S_2$ is given by $N_{S_i\cap S_j} = \{y\; |\; y = c_i\nabla h_i+c_j\nabla h_j, c_i, c_j\geq 0\}$. Since transversality holds for $(S_i, S_k)$ and $(S_j, S_k)$, we know that $\nabla h_i(x), \nabla h_k(x)$ and $\nabla h_j(x), \nabla h_k(x)$ are also linear independent. Thus, we obtain that $\nabla h_i(x), \nabla h_j(x), \nabla h_k(x)$ are linearly independent, and so, $y = c_i\nabla h_i+c_j\nabla h_j$ is linearly independent of $\nabla h_k$ for all $c_i, c_j>0$, i.e.,
\begin{align*}
     N_{S_i\cap S_j}(x) \cap (-N_{S_k}(x)) = \{0\},
\end{align*}
and hence, transversality holds for $S_i\cap S_j$ and $S_k$. Using the same set of arguments repeatedly, it is easy to show that transversality holds for $\cap_{i\neq j}S_i$ and $S_j$ for any $j$, and thus, \eqref{eq: transv S1 S2 SN} holds. \end{proof}

With the prior results, we are ready to present the following results on the strong invariance of $S_I$. Theorem \ref{thm:main2} and Corollary \ref{cor:main2}, presented below, are the multiple-set counterparts of Theorem \ref{thm:main1} and Corollary \ref{cor:main1}.

\begin{theorem}
    \label{thm:main2}
    Consider the system 
    \begin{align} 
    \label{eq:theoremsystem2}
    \dot{x}(t) \in \widehat{G}_{\hat{\gamma}}(x(t)).
    \end{align}
    Consider the set $S_I = \bigcap_{i = 1}^{N_h} S_i$ and suppose that the transversality condition holds for the pair $(S_i, S_j)$ for all $i, j\in I(x)$.
    Let $x(\cdot)$ be any trajectory of \eqref{eq:theoremsystem2} under a Lebesgue measurable control input $u(\cdot)$ with $x_0 = x(0) \in \textnormal{int}(S_I \cap \widehat{\Omega})$. Let $[0,T(x_0))$ be the (possibly empty) maximal interval such that $x(t) \in \textnormal{int}(\widehat{\Omega})$ for all $t \in [0,T(x_0))$. Then $x(t) \in S_I$ for all $t \in [0,T(x_0))$.
\end{theorem}

\begin{proof}
    Since transversality holds for all $x \in S_I$, by Lemma \ref{Lemma trans T cone} we have that $T_{S_I}(x) = \bigcap_{i=1}^{N_h} T_{S_i}(x)$ for all $x \in S_I$.
    The result then follows from Lemma \ref{lem:barG} using similar arguments as in Theorem \ref{thm:main1}.
\end{proof}

\begin{corr}
    \label{cor:main2}
    Under the hypotheses of Theorem \ref{thm:main1}, suppose there exists a bounded, open domain $D \subseteq \widehat{\Omega}$ such that $S_I \subset D$. If $x_0 = x(0) \in S_I$ then any Lebesgue measurable control input $u(t) \in \widehat{K}_{\hat{\gamma}}(x(t))$ renders the pair $(S_I,\widehat{G}_{\hat{\gamma}})$ strongly invariant; i.e. $x(t) \in S_I$ for all $t \geq 0$.
\end{corr}

\begin{proof}
       The result follows by using similar arguments as in Corollary \ref{cor:main1}.
\end{proof}


We now present an optimization problem which generates control inputs $u(t)$ which lie in the interior of $\widehat{K}(x)$. Define $z = \bmxs{v^T & \delta_1 & \ldots & \delta_{N_h}}^T$ and consider the optimization problem
{\small
\begin{subequations}\label{QP gen int}
\begin{align}
\min_{v, \delta_{1}, \delta_{2},\ldots , \delta_{N_h}} \; & C(z) \\
    \textrm{s.t.} \quad \quad  A_uv  \leq & b_u, \label{C1 cont const}\\
    L_{f_i}h_{s_i} + L_{g_i}h_{s_i}v \leq &-\delta_{i}h_{i},\quad i = 1, 2, \ldots, N_h\label{C3 safe const}
\end{align}
\end{subequations}}\normalsize
where 
$C:\mathbb R^{m+N_h}:\mathbb R$ is a convex objective function. Special cases of \eqref{QP gen int} include a simple feasibility problem, a linear program, and a quadratic program. Under the conditions of the results of this paper, any Lebesgue-measurable $u(t)$ computed from \eqref{QP gen int} will render the set $S_I$ invariant as per the results in Theorem \ref{thm:main2} and Corollary \ref{cor:main2}.

The following result provides guarantees on the feasibility of the optimization problem in \eqref{QP gen int}.

\begin{theorem}
\label{thm:last}
Suppose that the transversality condition holds for pair of any two sets $(S_{i}, S_{j})$ for all $i, j \in I(x)$, and $x\in \partial S_i\cap \partial S_j$. Then under Assumption \ref{assume:necessary}, the optimization problem \eqref{QP gen int} is feasible for all $x\in S_I$, and the set $\widehat{K}(x)$ from \eqref{eq:multiKforreal} has a non-empty interior for all $x\in S_I$.
\end{theorem}

\begin{proof}
Let $x\in \textnormal{int}(\bigcap S_{i})$. Then, we have that $h_i(x) \neq 0$ for all $i = 1,2, \dots, N_h$. Choose any $\bar v\in \mathcal U$ so that \eqref{C1 cont const} holds. Then, with this $\bar v$, define $\bar \delta_i = \frac{L_{f_i}h_{s_i} + L_{g_i}h_{s_i}v_i}{h_i}$, which is well-defined for all $x\in \textnormal{int}(\bigcap S_{i})$. Thus, we have that there exists a solution such that \eqref{C1 cont const}-\eqref{C3 safe const} holds, and so, the optimization problem \eqref{QP gen int} is feasible for all $x\in \textnormal{int}(\bigcap S_{i})$.

Now, with $\bar \delta_i$ defined as above, choose any $\hat \delta_i>\max\{0, \bar \sup_x|\delta_i(x)|\}$ so that $\hat \delta_i-\bar \delta_i>0$ for all $i\in 1, 2, \dots, N_h$. Then, with this choice of $\hat \delta_i$, we have that 
\begin{align*}
    L_{f_i}h_{s_i} + L_{g_i}h_{s_i}\bar v_i+\hat \delta_ih_i< L_{f_i}h_{s_i} + L_{g_i}h_{s_i}\bar v_i+\bar \delta_i h_i\leq 0,
\end{align*}
which implies that $(\bar v_i, \hat \delta_i)$ satisfy \eqref{C3 safe const} with strict inequality. Also, note that $\alpha_i(h_i) = \hat\delta_ih_i$ is an extended class $\mathcal K_\infty$ function for each $i = 1, 2, \dots, N_h$, since $\hat \delta_i> 0$. Thus, we have that there exists $u\in \mathcal U$ such that $L_fh_i + L_gh_iu<-\alpha_i(h_i)$, for any $x\in \textnormal{int}(\bigcap S_{i})$, and hence, $\widehat K(x)$ has a non-empty interior in that domain.

Next, we show that for any $x\in \bigcap \partial S_{i}$, the set $\widehat K$ has a non-empty interior. Under Assumption \ref{assume:necessary}, we have that there exists $u\in \mathcal U$ such that the inequalities \eqref{C3 safe const} strictly hold for all $i\in I(x)$ for all $x\in \bigcap \partial S_{i}$. For any $j\notin I(x)$, we have that $h_j(x)\neq 0$, and the analysis above guarantees that there exists a strict solution for \eqref{C3 safe const} for $j\notin I(x)$. 


Thus, we have that there exists a strict solution of \eqref{C3 safe const} for all $x\in \bigcap S_i$, this, $\textnormal{int}(\widehat K(x))$ is non-empty for all $x\in \bigcap S_i$. 
\end{proof}

From Theorem \ref{thm:last} we conclude that
when using strict control barrier functions
we can guarantee forward-invariance of multiple safe sets by solving the optimization problem \eqref{QP gen int} with additional non-negative slack variables in \eqref{C3 safe const}. By Theorem \ref{thm:last} the optimization problem is guaranteed to be feasible at all points $x \in \bigcap_{i=1}^{N_h} S_i$, and the non-emptiness of the set-valued map $\text{int}(\widehat{K}(x))$ is also guaranteed for all $x \in \bigcap_{i=1}^{N_h} S_i$, which in turn guarantees forward-invariance of the multiple safe sets.




\section{Conclusion}

In this paper we presented present a method to guarantee the forward invariance of composed sets using control barrier functions and incorporating input constraints. We demonstrated that control inputs rendering these sets invariant can be computed by solving a feasibility optimization problem. The computed control inputs are only required to be Lebesgue measurable and need not be continuous. We presented an optimization problem to compute these control inputs. Future work will incorporate more general control constraints and nonsmooth control barrier functions.

\section{Appendix}

\label{sec:diffincl}


A differential inclusion is a system with dynamics satisfying
\begin{align}
\label{eq:diffincl}
    \dot{x}(t) \in F(x)
\end{align}
where $F: \R^n \rarr \mathcal{P}(\R^n)$. 
Given a system $\dot{x} = f(t,x(t),u(t))$ where $u(t) \in \mathcal{U}(t,x)$ is a Lebesgue measurable function, all trajectories of the system can be considered simultaneously by defining the set-valued mapping $G : \R^n \rarr \mathcal{P}(\R^n)$ as
\begin{align}
\label{eq:inputdiffincl}
    G(t,x) = \{f(t,x,u) : u \in \mathcal{U}(t,x) \},
\end{align}
and considering the new differential inclusion $\dot{x}(t) \in G(t,x)$ \cite[Ch. 3, \S 15]{filippov2013differential}, \cite[Eq. (1.2)]{clarke1995qualitative} \cite[Ch. 10]{aubin2009set}.


\begin{define}
\label{def:locallyLip}
    Let $F : \R^n \rarr \mathcal{P}(\R^N)$ be a set-valued map. $F(x)$ is locally Lipschitz on a domain $D \subseteq \R^n$ if every point $x \in D$ admits a neighborhood $U = U(x)$ and a positive constant $L = L(x)$ such that 
    \begin{align}
        x_1, x_2 \in U \implies F(x_2) \subseteq F(x_1) + L\nrm{x_1 - x_2}B(0,1),
    \end{align}
    where $B(0,1)$ denotes the closed unit ball in $\R^n$. 
\end{define}
We point out that $F(x)$ being locally Lipschitz implies that $F(x)$ is upper semicontinuous \cite{cortes2008discontinuous}.


\begin{define}
    Let $S \subseteq \R^n$. The projection operator from $x_0 \in \R^n$ onto $S$ is defined as
    \begin{align}
        \prj_S(x_0) = \underset{x \in S}{\arg\min} \nrm{x - x_0}.
    \end{align}
\end{define}

\begin{define}[\cite{clarke2008nonsmooth}]
\label{def:NPS}
    The proximal normal cone of the set $S$ at $x$, denoted $N^P_S(x) \subseteq \R^n$, is defined as
    \begin{align*}
        N^P_S(x) = \{\theta (v-x) \in \R^n : \theta \geq 0,\ v \not\in S,\ \prj_S(v) = x \}.
    \end{align*}
    By convention, it always holds that $\{0\} \in N^P_S(s')$.
\end{define}

\begin{define}[\cite{clarke1995qualitative}]
\label{def:tangentcone}
    The Clarke tangent cone of the set $S$ at $x$, denoted $T_S(x)$, is defined as
    \begin{align}
    \label{eq:TS}
        T_S(x) = \brc{v \in \R^n : \underset{\substack{y \rarr x \\ t \downarrow 0}}{\lim\sup} \frac{d_S(y + tv) - d_S(y)}{t} \leq 0 }.
    \end{align}
\end{define}

\begin{lemma}
\label{lem:cones}
Let $S$ be defined as $S = \{x : h(x) \leq 0 \}$, where $h : \R^n \rarr \R$ is continuously differentiable. Then all of the following statements hold:
\begin{itemize}
    \item For all $x \in \textnormal{int}(S)$, $T_S(x) = \R^n$ and $N^P_S(x) = \{0\}$.
    \item For all $x \in \partial S$, the proximal normal cone satisfies
    \begin{align}
        N^P_S(x) = T_S(x)^{\circ} = \brc{\theta \frac{\partial h(x)}{\partial x} : \theta \geq 0},
    \end{align}
    i.e. $N^P_S(x)$ and $T_S(x)$ are polar to each other.
    \item For all $x \in \partial S$, the tangent cone $T_S(x)$ satisfies
    \begin{align}
    \label{eq:TSdef}
        T_S(x) = \brc{ v \in \R^n : \ip{v, \frac{\partial h(x)}{\partial x}} \leq 0 }.
    \end{align}
\end{itemize}
\end{lemma}

\begin{proof}
    The first statement trivially holds for $N^P_S(x)$ since $\prj(u)$ for $u \not\in S$ will always yield a point on the boundary $\partial S$ \cite[p. 22]{clarke2008nonsmooth}. In addition, $T_S(x) = \R^n$ $\forall x \in \textnormal{int}(S)$ follows from \eqref{eq:TS} and noting that for all such $x$ there exists an open neighborhood $U(x)$ with $x \in U(x) \in S$. 
    
    Consider all $x \in \partial S$.
    In smooth manifolds $N^P_S(x)$ coincides with the normal space \cite[p. 9]{clarke2008nonsmooth}. Since $h$ is continuously differentiable, the boundary $\partial S = \{x : h(x) = 0\}$ is a smooth manifold with normal space $\{\hat{\theta} \frac{\partial h(x)}{\partial x},\ \hat{\theta} \in \R \}$. It follows from Definition \ref{def:NPS} that only the vectors $\{ \theta \frac{\partial h(x)}{\partial x} : \theta \geq 0 \}$ which point outside the set $S$ are in the proximal normal cone to $S$. To show that $N^P_S(x) = T_S(x)^{\circ}$ note that $T_S(x)^{\circ} = N_S(x)$ \cite[Ch. 5]{clarke2008nonsmooth} and in a finite dimensional Hilbert space we have $N_S(x) = N^P_S(x)$ when $N^P_S(x)$ is closed \cite[Theorem 6.1]{clarke2008nonsmooth}.
    
    Finally, the characterization of $T_S(x)$ in \eqref{eq:TSdef} follows from the fact that $T_S(x) = N_S(x)^{\circ}$ \cite[Prop. 5.4]{clarke2008nonsmooth}, which in this case satisfies $N_S(x)^{\circ} = N^P_S(x)^{\circ}$.
\end{proof}

\begin{lemma}
\label{lem:setmult}
    Let $v \in \R^{m \times p}$. Then
    $\bar{B}^{n \times m}(0,1) v \subseteq \nrm{v} \bar{B}^{n \times p}(0,1)$. 
\end{lemma}

\begin{proof}
    The equality clearly holds when $v = \bm 0$.
    Choose any $u \in \bar{B}^{n \times m}(0,1)$, and note that $u v \in \R^{n \times p}$.
    Define $\hat{v} = \frac{v}{\nrm{v}}$ for $v \neq \bm 0$.
    Then $uv = \nrm{v} u \hat{v}$. Note that $u \hat{v} \in \R^{n \times p}$ and $\nrm{u \hat{v}} \leq \nrm{u} \nrm{\hat{v}} = 1$, implying that $u \hat{v} \in \bar{B}^{n \times p}(0,1)$ and therefore $uv = \nrm{v} u \hat{v} \in \nrm{v} \bar{B}^{n \times p}(0,1)$. This implies $\bar{B}^{n \times m}(0,1) v \subseteq \nrm{v} \bar{B}^{n \times p}(0,1)$.
\end{proof}

\begin{lemma}
\label{lem:setmatrixprod}
    Let $S \subset \R^{m \times p}$. Then $\bar{B}^{{n \times m}}(0,1)S \subseteq \nrm{v_{\max}} \bar{B}^{n \times p}(0,1)$, where $v_{\max} = \pth{\arg\sup_{v \in S} \nrm{v}}$.
\end{lemma}

\begin{proof}
    By the definition of the set-valued matrix product (Section \ref{sec:notation}), we have
    \begin{align*}
        \bar{B}^{{n \times m}}(0,1)S &= \{uv \in \R^{n \times p} : u \in  \bar{B}^{{n \times m}}(0,1),\ v \in S\}
    \end{align*}
    Using Lemma \ref{lem:setmult}, this can equivalently be written as
    \begin{align}
        \bar{B}^{{n \times m}}(0,1)S &= \bigcup_{v \in S} \bar{B}^{{n \times m}}(0,1) v \\
        &\subseteq \bigcup_{v \in S} \nrm{v} \bar{B}^{{n \times p}}(0,1) \label{eq:cupnorm}
    \end{align}
    Let $v_{\max} = \pth{\arg\sup_{v \in S} \nrm{v}}$. Then it holds that $\nrm{v} \bar{B}^{n \times p}(0,1) \subseteq \nrm{v_{\max}} \bar{B}^{n \times p}(0,1)$ for all $v \in S$, which with \eqref{eq:cupnorm} implies that $\bar{B}^{{n \times m}}(0,1)S \subseteq \nrm{v_{\max}} \bar{B}^{n \times p}(0,1)$.
\end{proof}

\begin{corr}
    The following holds:
    \begin{align}
        \bar{B}^{n \times m}(0,1) \bar{B}^{m \times p}(0,1) \subseteq \bar{B}^{n \times p}(0,1)
    \end{align}
\end{corr}

\begin{proof}
    Follows directly from Lemma \ref{lem:setmatrixprod} by noting that $\sup_{v \in \bar{B}^{m \times p}(0,1)} \nrm{v} = 1$.
\end{proof}

\begin{prop}[Adapted from {\cite[Prop. 2.14]{freeman2008robust}}]
\label{prop:Freeman}
    Let $G : \R^n \rarr \mathcal{P}(\R^m)$ be locally Lipschitz continuous with nonempty, compact, convex values. Let $\phi : \R^n \times \R^m \rarr \R$ be locally Lipschitz continuous and such that the mapping $z \mapsto \phi(x,z)$ is convex for each fixed $x \in \R^n$. Then $F : \R^n \rarr \mathcal{P}(\R^m)$ defined by $F(x) = \{z \in G(x) : \phi(x,z) < 0 \}$ is locally Lipschitz continuous on $\textup{Dom}(F)$.
\end{prop}

\begin{lemma}
    \label{lem:twolocalLipschitz}
    Consider a function $f(x_1,x_2)$ with $f : \R^n \times \R^m \rarr \R$ and let $D_1 \subset \R^n$, $D_2 \subset \R^m$ be open subsets. Suppose that for all fixed $x_1 \in D_1$ the function $f(x_1,\cdot)$ is locally Lipschitz on $D_2$, and for all fixed $x_2 \in D_2$ the function $f(\cdot,x_2)$ is locally Lipschitz on $\R^n$. Then $f$ is locally Lipschitz on $D_1 \times D_2$.
\end{lemma}

\bibliographystyle{IEEEtran}
\bibliography{KCBF.bib}

\end{document}